\documentclass[12pt,twoside]{amsart}
\usepackage{amssymb,amsmath,amsthm, amscd, enumerate, mathrsfs}
\usepackage{graphicx, hhline}
\usepackage[all]{xy}
\usepackage[usenames]{color}
\usepackage{hyperref}
\usepackage{fancyhdr}
\hypersetup{colorlinks=true}

\title{Non-vanishing theorem for lc pairs admitting a Calabi--Yau pair}
\author{Kenta Hashizume}
\date{2017/8/3, version 0.03}
\address{Department of Mathematics, Graduate School of Science, 
Kyoto University, Kyoto 606-8502, Japan}
\email{hkenta@math.kyoto-u.ac.jp}

\newtheorem{thm}{Theorem}[section]

\newtheorem{lem}[thm]{Lemma}

\newtheorem{conj}[thm]{Conjecture}

\theoremstyle{definition}
\newtheorem{defn}[thm]{Definition}
\newtheorem{rem}[thm]{Remark}

\newtheorem*{ack}{Acknowledgments} 
\newtheorem{say}[thm]{}

\newtheorem{step}{Step}

\newtheorem{step3}{Step}

\begin{document}

\maketitle

\begin{abstract}
We prove the non-vanishing conjecture for lc pairs $(X,\Delta)$ when $X$ is of Calabi--Yau type. 
\end{abstract}

\tableofcontents

\section{Introduction}\label{sec1}
Throughout this paper we will work over the complex number field.

In this paper we deal with varieties of Calabi--Yau type.

\begin{defn}
Let $X$ be a normal projective variety. 
Then $X$ is {\em of Calabi--Yau type} if there is an $\mathbb{R}$-divisor $C\geq0$ such that $(X,C)$ is lc and $K_{X}+C\equiv 0$.  
\end{defn}

The main result of this paper is the non-vanishing theorem for lc pairs whose underlying variety is of Calabi--Yau type.

\begin{thm}\label{thmmain}
Let $X$ be a normal projective variety. 
Suppose that $X$ is of Calabi--Yau type. 

Then, for any lc pair $(X,\Delta)$, the non-vanishing conjecture holds. 
In other words, if $K_{X}+\Delta$ is pseudo-effective there exists an $\mathbb{R}$-divisor $E\geq0$ such that $K_{X}+\Delta\sim_{\mathbb{R}}E$. 
\end{thm}

Here we recall statement of the non-vanishing conjecture. 

\begin{conj}[Non-vanishing]\label{conjnon}
Let $(X,\Delta)$ be a projective lc pair such that $K_{X}+\Delta$ is pseudo-effective. 
Then there exists an $\mathbb{R}$-divisor $E\geq0$ such that $K_{X}+\Delta \sim_{\mathbb{R}}E$. 
\end{conj}

Conjecture \ref{conjnon} is one of the most important open problems in the minimal model theory. 
It is known by Birkar \cite{birkar-existII} that Conjecture \ref{conjnon} implies the minimal model conjecture. 
Today Conjecture \ref{conjnon} is known for lc pairs of dimension $\leq3$ but the conjecture is only partially solved in higher dimensional case.  
For example, Conjecture \ref{conjnon} for lc pairs $(X,\Delta)$ of ${\rm dim}\,X\geq4$ is known when 
\begin{itemize}
\item
$(X,\Delta)$ is klt and $\Delta$ is big (cf.~\cite{bchm}), 
\item
$(X,\Delta)$ is klt and $X$ is rationally connected (cf.~\cite{gongyo-nonvanishing}), or
\item
$K_{X}\equiv0$ (cf.~\cite{gongyo}, see also \cite{ckp}, \cite{kawamata}, \cite{ambro} and \cite{nakayama-zariski-decom}). 
\end{itemize}
Moreover the arguments in \cite{gongyo-nonvanishing} and \cite{dhp} show that Conjecture \ref{conjnon} holds for any lc pair $(X,\Delta)$ such that ${\rm dim}\,X=4$ and $X$ is uniruled, though it is not written explicitly in their papers. 
Lazi\'c and Peternell proved Conjecture \ref{conjnon} for terminal $4$-folds under the assumption that $\chi(X,\mathcal{O}_{X})\neq0$ and $K_{X}$ has a singular metric with algebraic singularities and positive curvature current (cf.~\cite[Theorem B]{lazicpeter}).

We note that the case $K_{X}\equiv0$ mentioned above is a special case of Theorem \ref{thmmain}. 
Indeed, when $K_{X}\equiv0$ in Theorem \ref{thmmain}, the statement of the theorem is equivalent to the abundance theorem for numerically trivial lc pairs and it is proved by Gongyo \cite{gongyo} (see also \cite{ckp} and \cite{kawamata}). 
Therefore, in view of Conjecture \ref{conjnon}, Theorem \ref{thmmain} can be regarded as a generalization of the result of \cite{gongyo}.


The contents of this paper are as follows: 
In Section \ref{sec2} we collect some notations, definitions and important theorems. 
In Section \ref{sec3} we prove Theorem \ref{thmmain}. 

\begin{ack}
The author was partially supported by JSPS KAKENHI Grant Number JP16J05875 from JSPS. The author thanks Professor Osamu Fujino for discussions and warm encouragement. 
He also thanks Professors Yoshinori Gongyo and Yusuke Nakamura for comments.  
\end{ack}

\section{Preliminaries}\label{sec2}
In this section we collect notations, definitions and some important theorems. 

\begin{say}[Singularities of pairs]
A {\em pair} $(X,\Delta)$ consists of a normal variety $X$ and a boundary $\mathbb{R}$-divisor $\Delta$, that is, an $\mathbb{R}$-divisor whose coefficients belong to $[0,1]$, on $X$ such that $K_{X}+\Delta$ is $\mathbb{R}$-Cartier. 

Let $(X,\Delta)$ be a pair and let $D$ be a prime divisor over $X$. 
Then $a(D,X,\Delta)$ denotes the discrepancy of $D$ with respect to $(X,\Delta)$. 
In this paper we use the definitions of Kawamata log terminal (klt, for short) pair, log canonical (lc, for short) pair and divisorially log terminal (dlt, for short) pair written in \cite{kollar-mori} or \cite{bchm}. 
\end{say}

Next we define some models.

\begin{defn}[Log birational model]\label{deflogbir}
Let $\pi\!:\!X \to Z$ be a projective morphism from a normal variety to a variety and let $(X,\Delta)$ be an lc pair.  
Let $\pi'\!:\!X' \to Z$ be a projective morphism from a normal variety to $Z$ and let $\phi\!:\!X \dashrightarrow X'$ be a birational map over $Z$. 
Let $E$ be the reduced $\phi^{-1}$-exceptional divisor on $X'$, that is, $E=\sum E_{j}$ where $E_{j}$ are $\phi^{-1}$-exceptional prime divisors on $X'$. 
Then $(X', \Delta'=\phi_{*}\Delta+E)$ is called a {\em log birational model} of $(X,\Delta)$ over $Z$. 
\end{defn}

\begin{defn}[Log minimal model and Mori fiber space]\label{deflogmin}
Notations as in Definition \ref{deflogbir}, a log birational model $(X', \Delta')$ of $(X,\Delta)$ over $Z$ is a {\em weak log canonical model} ({\em weak lc model}, for short) if 
\begin{itemize}
\item
$K_{X'}+\Delta'$ is nef over $Z$, and 
\item
for any prime divisor $D$ on $X$ which is exceptional over $X'$, we have
$$a(D, X, \Delta) \leq a(D, X', \Delta').$$ 
\end{itemize}
A weak lc model $(X',\Delta')$ of $(X,\Delta)$ over $Z$ is a {\em log minimal model} if 
\begin{itemize}
\item
$(X',\Delta')$ is $\mathbb{Q}$-factorial, and 
\item
the above inequality on discrepancies is strict. 
\end{itemize}
A log minimal model $(X',\Delta')$ of $(X, \Delta)$ over $Z$ is called a {\em good minimal model} if $K_{X'}+\Delta'$ is semi-ample over $Z$.    

On the other hand, a log birational model $(X', \Delta')$ of $(X,\Delta)$ over $Z$ is called a {\em Mori fiber space} if $X'$ is $\mathbb{Q}$-factorial and there is a contraction $X' \to W$ with ${\rm dim}\,W<{\rm dim}\,X'$ such that 
\begin{itemize}
\item
the relative Picard number $\rho(X'/W)$ is one and $-(K_{X'}+\Delta')$ is ample over $W$, and 
\item
for any prime divisor $D$ over $X$, we have
$$a(D,X,\Delta)\leq a(D,X',\Delta')$$
and strict inequality holds if $D$ is a divisor on $X$ and exceptional over $X'$.
\end{itemize} 
\end{defn}

\begin{defn}[Log smooth model]\label{deflogsm}
Let $(X,\Delta)$ be an lc pair and let $f\!:\!Y \to X$ be a log resolution of $(X,\Delta)$. 
Let $\Gamma$ be a boundary $\mathbb{R}$-divisor on $Y$ such that $(Y,\Gamma)$ is log smooth. 
Then $(Y,\Gamma)$ is a {\em log smooth model} of $(X,\Delta)$ if we can write 
$$K_{Y}+\Gamma=f^{*}(K_{X}+\Delta)+F$$
with an effective $f$-exceptional divisor $F$ such that every $f$-exceptional prime divisor $E$ satisfying $a(E,X,\Delta)>-1$ is a component of $F$ and $\Gamma-\llcorner \Gamma \lrcorner$.  
\end{defn}

Our definition of log minimal model and Mori fiber space is slightly different from that of \cite{birkar-flip}. 
The difference is that we do not assume those models to be dlt. 
But this difference is intrinsically not important (see \cite[Remark 2.7]{has-trivial}). 
In our definition, any weak lc model $(X',\Delta')$ of a $\mathbb{Q}$-factorial lc pair $(X,\Delta)$ constructed with the $(K_{X}+\Delta)$-MMP
is a log minimal model of $(X,\Delta)$ even though $(X',\Delta')$ may not be dlt. 

The following theorem proved by Birkar \cite{birkar-flip} is frequently implicitly used in this paper.

\begin{thm}[cf.~{\cite[Theorem 4.1]{birkar-flip}}]\label{thmtermi}
Let $(X,\Delta)$ be a $\mathbb{Q}$-factorial lc pair such that $(X,0)$ is klt, and let $\pi\!:\!X \to Z$ be a projective morphism of normal quasi-projective varieties. 
If there exists a log minimal model of $(X,\Delta)$ over $Z$, then any $(K_{X}+\Delta)$-MMP over $Z$ with scaling of an ample divisor terminates. 
\end{thm}

Next we recall definition of pseudo-effective threshold. 

\begin{defn}
Let $(X,\Delta)$ be a projective lc pair and let $M\geq0$ be an $\mathbb{R}$-Cartier $\mathbb{R}$-divisor such that $K_{X}+\Delta+M$ is pseudo-effective. 
Then the {\em pseudo-effective threshold} of $M$ with respect to $(X,\Delta)$, denoted by  $\tau(X,\Delta;M)$, is
$$\tau(X,\Delta;M)={\rm inf}\{t\in \mathbb{R}_{\geq 0}\mid K_{X}+\Delta+tM {\rm \; is \; pseudo\mathchar`-effective}\}.$$ 
\end{defn}

We close this section with two important theorems proved by Hacon, M\textsuperscript{c}Kernan and Xu \cite{hmx-acc}. 

\begin{thm}[{cf.~\cite[Theorem 1.1]{hmx-acc}}]\label{thmacclct}
Fix a positive integer $n$, a set $I \subset [0,1]$ and a set $J\subset \mathbb{R}_{>0}$, where $I$ and $J$ satisfy the DCC. 
Let $\mathfrak{T}_{n}(I)$ be the set of lc pairs $(X,\Delta)$, where $X$ is a variety of dimension $n$ and the coefficients of $\Delta$ belong to $I$. 
Then the set 
$$\{ {\rm lct}(X,\Delta;M) \mid (X,\Delta) \in \mathfrak{T}_{n}(I), {\rm \; the \;coefficients \; of\; }M{\rm \;belong\; to\;}J \}$$
satisfies the ACC, where ${\rm lct}(X,\Delta;M)$ is the log canonical threshold of $M$ with respect to $(X,\Delta)$.
\end{thm}

\begin{thm}[{cf.~~\cite[Theorem D]{hmx-acc}}]\label{thmglobalacc}
Fix a positive integer $n$ and a set $I \subset [0,1]$, which satisfies the DCC. 

Then there is a finite set $I_{0}\subset I$ with the following property:

If $(X,\Delta)$ is an lc pair such that 
\begin{enumerate}
\item[(i)]
X is projective of dimension $n$, 
\item[(ii)]
the coefficients of $\Delta$ belong to $I$, and
\item[(iii)]
$K_{X}+\Delta$ is numerically trivial,
\end{enumerate}
then the coefficients of $\Delta$ belong to $I_{0}$. 
\end{thm}

\section{Proof of Theorem \ref{thmmain}}\label{sec3}

In this section we prove Theorem \ref{thmmain}. 

\begin{lem}\label{lemdeform}
Let $(X,B)$ be a projective lc pair. 
Let $\pi\!:\!(X,B)\to Z$ be a contraction to a normal projective variety $Z$ such that $K_{X}+B\sim_{\mathbb{R}}\pi^{*}D$ for some $D$ on $Z$. 
Then we can construct the following diagram  
$$
\xymatrix{
(X,B)\ar@{-->}[r] \ar[d]_{\pi}&(X_{0},B_{0})\ar^{\pi_{0}}[d]\\ 
Z&Z_{0}\ar^{h}[l]
}
$$
such that
\begin{itemize}\item$\pi_{0}$ and $h$ are contractions and $h$ is birational, \item$(X_{0},B_{0})$ is a log birational model of $(X,B)$ and it is a projective $\mathbb{Q}$-factorial lc pair such that $(X_{0},0)$ is klt,  \item$K_{X_{0}}+B_{0}\sim_{\mathbb{R}}\pi_{0}^{*}h^{*}D$, \item$Z_{0}$ is a projective $\mathbb{Q}$-factorial variety such that $(Z_{0},0)$ is klt, and\item$B_{0}=B'_{0}+B''_{0}$ with $B'_{0}\geq0$ and $B''_{0}\geq0$ such that $B''_{0}\sim_{\mathbb{R},\,Z_{0}}0$ and any lc center of $(X_{0},B'_{0})$ dominates $Z_{0}$. \end{itemize} 
\end{lem}

\begin{proof}
The idea of the proof can be found in \cite[proof of Lemma 4.5]{has-trivial}. 
We prove Lemma \ref{lemdeform} with three steps. 
\begin{step}\label{step1}
In this step we construct a diagram 
$$
\xymatrix{
(X,B)\ar_{\pi}[d] \ar@{-->}[r]&(\overline{X},\overline{B})\ar^{\overline{\pi}}[d]\\ 
Z&\overline{Z} \ar^{\overline{h}}[l]
}
$$
such that 
\begin{enumerate}
\item
$\overline{\pi}$ and $\overline{h}$ are contractions and $\overline{h}$ is birational, 
\item
$(\overline{X},\overline{B})$ is a log birational model of $(X,B)$ and it is a projective $\mathbb{Q}$-factorial lc pair such that $(\overline{X},0)$ is klt,  
\item 
$\overline{B}=\overline{B}'+\overline{B}''$ with $\overline{B}'\geq0$ and $\overline{B}''\geq0$ such that $\overline{B}''\sim_{\mathbb{R},\,\overline{Z}}0$ and any lc center of $(\overline{X},\overline{B}')$ dominates $\overline{Z}$, and 
\item 
$K_{\overline{X}}+\overline{B}\sim_{\mathbb{R}}\overline{\pi}^{*}\overline{h}^{*} D$. 
\end{enumerate}

First take a dlt blow-up $(W,\Psi)\to (X,B)$ as in \cite[Corollary 2.14]{has-trivial}. 
Then we can  decompose $\Psi=\Psi'+\Psi''$ with $\Psi'\geq0$ and $\Psi''\geq 0$ such that $\Psi''$ is vertical over $Z$  and any lc center of $(W,\Psi')$ dominates $Z$. 
Moreover we have $K_{W}+\Psi'+\Psi''\sim_{\mathbb{R},\,Z}0$. 
Since $(W,\Psi')$ is $\mathbb{Q}$-factorial and dlt, by \cite[Theorem 1.1]{has-mmp}, we can run the $(K_{W}+\Psi')$-MMP over $Z$ with scaling and get a good minimal model $(W,\Psi')\dashrightarrow (\overline{X},\overline{B}')$ over $Z$. 
Let $\overline{B}$ and $\overline{B}''$ be the birational transform of $\Psi$ and $\Psi''$ on $\overline{X}$ respectively. 
Then $\overline{B}=\overline{B}'+\overline{B}''$. 
Let $\overline{\pi}\!:\!\overline{X}\to \overline{Z}$ be the contraction over $Z$ induced by $K_{\overline{X}}+\overline{B}'$, and let $\overline{h}\!:\!\overline{Z}\to Z$ be the induced morphism. 

We can easily check that $(\overline{X},\overline{B}=\overline{B}'+\overline{B}'')$, $\overline{\pi}\!:\!\overline{X}\to \overline{Z}$ and $\overline{h}\!:\!\overline{Z}\to Z$ satisfy conditions (1), (2), (3) and (4). 
Indeed, it is easy to see that $\overline{\pi}$ and $\overline{h}$ satisfy condition (1). 
We also have $K_{\overline{X}}+\overline{B}\sim_{\mathbb{R}}\overline{\pi}^{*}\overline{h}^{*}D$, which is condition (4).   
Moreover, since $(X,B)$ is lc and since $K_{X}+B$ and $K_{\overline{X}}+\overline{B}$ are both $\mathbb{R}$-linearly equivalent to the pullback of $D$, we see that $(\overline{X},\overline{B})$ is lc. 
Now it is clear that $(\overline{X},\overline{B})$ satisfies condition (2). 
It is also clear that $\overline{B}''\sim_{\mathbb{R},\,\overline{Z}}0$ because  $K_{\overline{X}}+\overline{B}'\sim_{\mathbb{R},\,\overline{Z}}0$.
Finally we check that any lc center of $(\overline{X},\overline{B}')$ dominates $\overline{Z}$.
Pick any prime divisor $P$ over $\overline{X}$ such that $a(P,\overline{X},\overline{B}')=-1$.  
Then $a(P,W,\Psi')=-1$ and thus $P$ dominates $Z$. 
Since $\overline{h}\!:\!\overline{Z}\to Z$ is birational, we see that $P$ dominates $\overline{Z}$. 
Therefore any lc center of $(\overline{X},\overline{B}')$ dominates $\overline{Z}$, and we see that $(\overline{X},\overline{B}=\overline{B}'+\overline{B}'')$ satisfies condition (3). 
So we complete this step.
\end{step} 

\begin{step}\label{step2}
We put $\overline{D}=\overline{h}^{*}\!D$. 
Then $K_{\overline{X}}+\overline{B}\sim_{\mathbb{R}}\overline{\pi}^{*}\overline{D}$ by construction. 
In this step we construct a diagram 
$$
\xymatrix{
(\overline{X},\overline{B})\ar_{\overline{\pi}}[d] \ar@{-->}[r]&(X_{0},B_{0})\ar^{\pi_{0}}[d]\\ 
\overline{Z}&Z_{0} \ar^{h_{0}}[l]
}
$$
with a projective $\mathbb{Q}$-factorial variety $Z_{0}$ such that  $(Z_{0},0)$ is klt and
\begin{enumerate}
\item[(1$'$)]
$\pi_{0}$ and $h_{0}$ are contractions and $h_{0}$ is birational, 
\item[(2$'$)]
$(X_{0},B_{0})$ is a log birational model of $(\overline{X},\overline{B})$ and it is a projective $\mathbb{Q}$-factorial lc pair such that $(X_{0},0)$ is klt, 
\item[(3$'$)]
$B_{0}=B_{0}'+B_{0}''$ with $B_{0}'\geq0$ and $B_{0}''\geq0$ such that $B''_{0}\sim_{\mathbb{R},\,Z_{0}}0$ and any lc center of $(X_{0},B'_{0})$ dominates $Z_{0}$, and 
\item[(4$'$)] 
$K_{X_{0}}+B_{0}\sim_{\mathbb{R}} \pi_{0}^{*}h_{0}^{*}\overline{D}$. 
\end{enumerate}

By condition (3) in Step \ref{step1}, there exists an $\mathbb{R}$-divisor $\overline{T}\geq0$ on $\overline{Z}$ such that $\overline{B}''\sim_{\mathbb{R}}\overline{\pi}^{*}\overline{T}$.  
By condition (3) in Step \ref{step1} and \cite[Corollary 3.2]{fg-bundle}, there exists a klt pair on $\overline{Z}$. 
Let $h_{0}\!:\!Z_{0}\to \overline{Z}$ be a dlt blow-up of the klt pair. 
Then $h_{0}$ is a small birational morphism and $Z_{0}$ is $\mathbb{Q}$-factorial. 
Let $\overline{\varphi}\!:\!\overline{W}\to \overline{X}$ be a log resolution of $(\overline{X},\overline{B}')$ such that the induced map $\pi_{\overline{W}}\!:\!\overline{W}\dashrightarrow Z_{0}$ is a morphism. 
We pick a boundary divisor $\Psi'_{\overline{W}}$ so that $(\overline{W},\Psi'_{\overline{W}})$ is a log smooth model of $(\overline{X},\overline{B}')$. 
Then we have
\begin{equation*}
\begin{split}
K_{\overline{W}}+\Psi'_{\overline{W}}=&\overline{\varphi}^{*}(K_{\overline{X}}+\overline{B}')+E_{\overline{W}}\sim_{\mathbb{R}}\overline{\varphi}^{*}\overline{\pi}^{*}(\overline{D}-\overline{T})+E_{\overline{W}}\\
=&(h_{0}\circ\pi_{\overline{W}})^{*}(\overline{D}-\overline{T})+E_{\overline{W}}
\end{split}
\end{equation*}
for a $\overline{\varphi}$-exceptional divisor $E_{\overline{W}}\geq0$. 
By construction of $\Psi'_{\overline{W}}$, for any $\overline{\varphi}$-exceptional prime divisor $E_{i}$ on $\overline{W}$, $E_{i}$ is a component of $E_{\overline{W}}$ if and only if $a(E_{i},\overline{X},\overline{B}')>-1$. 

We run the $(K_{\overline{W}}+\Psi'_{\overline{W}})$-MMP over $Z_{0}$ with scaling. 
By the argument of very exceptional divisors (cf.~\cite[Theorem 3.5]{birkar-flip}), after finitely many steps, $E_{\overline{W}}$ is contracted and thus we get a model $(\overline{W},\Psi'_{\overline{W}})\dashrightarrow(X_{0},B'_{0})$ such that 
$K_{X_{0}}+B_{0}'\sim_{\mathbb{R},\,Z_{0}}0$. 
Let $\pi_{0}\!:\!X_{0}\to Z_{0}$ be the induced morphism. 
Now we have the following diagram. 
$$ \xymatrix{ (\overline{X},\overline{B}')\ar_{\overline{\pi}}[d]&(\overline{W},\Psi'_{\overline{W}})\ar_{\overline{\varphi}}[l] \ar_{\pi_{\overline{W}}}[dr]\ar@{-->}[r]&(X_{0},B'_{0})\ar^{\pi_{0}}[d]\\ \overline{Z}&&Z_{0} \ar^{h_{0}}[ll]}$$
Moreover we have $K_{X_{0}}+B'_{0}\sim_{\mathbb{R}}\pi_{0}^{*}h_{0}^{*}(\overline{D}-\overline{T})$. 
Let $B''_{0}$ be the birational transform of $\overline{\varphi}^{*}\overline{B}''$ on $X_{0}$, and we put $B_{0}=B'_{0}+B''_{0}$. 
Recall that the divisor $\overline{T}$ on $\overline{Z}$ satisfies $\overline{B}''\sim_{\mathbb{R}}\overline{\pi}^{*}\overline{T}$.  

From now on we check that  $(X_{0}, B_{0}=B'_{0}+B''_{0})$, $\pi_{0}\!:\!X_{0}\to Z_{0}$ and $h_{0}\!:\!Z_{0}\to \overline{Z}$ satisfy conditions (1$'$), (2$'$), (3$'$) and (4$'$). 
It is clear that $\pi_{0}$ and $h_{0}$ satisfy condition (1$'$). 
Moreover, since $B''_{0}\sim_{\mathbb{R}}\pi_{0}^{*}h_{0}^{*}\overline{T}$, we have
$$
K_{X_{0}}+B_{0}=K_{X_{0}}+B'_{0}+B''_{0}\sim_{\mathbb{R}}\pi_{0}^{*}h_{0}^{*}(\overline{D}-\overline{T})+\pi_{0}^{*}h_{0}^{*}\overline{T}=\pi_{0}^{*}h_{0}^{*}\overline{D}.
$$
Therefore $K_{X_{0}}+B_{0}$ satisfies condition (4$'$). 
Next pick any prime divisor $P$ over $X_{0}$ such that $a(P,X_{0},B'_{0})=-1$. 
Then $a(P,\overline{W},\Psi'_{\overline{W}})=-1$, and hence $a(P,\overline{X},\overline{B}')=-1$ because $(\overline{W},\Psi'_{\overline{W}})$ is a log smooth model of $(\overline{X},\overline{B}')$ (cf.~\cite[Remark 2.11]{has-mmp}). 
So $P$ dominates $\overline{Z}$ by condition (3) in Step \ref{step1}. 
Since $h_{0}\!:\!Z_{0}\to \overline{Z}$ is birational, $P$ dominates $Z_{0}$ and hence we see that any lc center of $(X_{0},B'_{0})$ dominates $Z_{0}$. 
Now we can easily check that $(X_{0}, B_{0}=B'_{0}+B''_{0})$ satisfies condition (3$'$). 
Finally we check condition (2$'$). 
We only check that $(X_{0},B_{0})$ is a log birational model of $(\overline{X},\overline{B})$ because others are easy. 
Note that $(X_{0},B_{0})$ is lc since $(\overline{X},\overline{B})$ is lc and since $K_{\overline{X}}+\overline{B}$ and $K_{X_{0}}+B_{0}$ are both $\mathbb{R}$-linearly equivalent to the pullback of $\overline{D}$. 
Let $E_{i}$ be a $\overline{\varphi}$-exceptional prime divisor on $\overline{W}$ such that $a(E_{i},\overline{X},\overline{B})>-1$. 
We show that $E_{i}$ is contracted by $\overline{W}\dashrightarrow X_{0}$. 
Since $a(E_{i},\overline{X},\overline{B}')\geq a(E_{i},\overline{X},\overline{B})>-1$ we see that $E_{i}$ is a component of $E_{\overline{W}}$. 
Then $E_{i}$ is contracted by $\overline{W}\dashrightarrow X_{0}$ since $E_{\overline{W}}$ is contracted by $\overline{W}\dashrightarrow X_{0}$. 
In this way we see that $(X_{0}, B_{0})$ is a log birational model of $(\overline{X},\overline{B})$. 
So $(X_{0},B_{0})$ satisfies condition (2$'$) and we complete this step. 
\end{step}

\begin{step}\label{step3}
Now we have constructed the following diagram 
$$
\xymatrix{
(X,B)\ar_{\pi}[d] \ar@{-->}[r]&(\overline{X},\overline{B})\ar_{\overline{\pi}}[d] \ar@{-->}[r]&(X_{0},B_{0})\ar^{\pi_{0}}[d]\\ 
Z_{Y}&\overline{Z} \ar^{\overline{h}}[l]&Z_{0} \ar^{h_{0}}[l]
}
$$
satisfying conditions (1), (2), (3) and (4) in Step \ref{step1} and (1$'$), (2$'$), (3$'$) and (4$'$) in Step \ref{step2}, and furthermore $Z_{0}$ is $\mathbb{Q}$-factorial and $(Z,0)$ is klt. 
We set $h=\overline{h}\circ h_{0}:Z_{0}\to Z$. 
By construction $h$ is birational, and it is clear that the following 
$$
\xymatrix{
(X,B)\ar_{\pi}[d] \ar@{-->}[r]&(X_{0},B_{0}=B'_{0}+B''_{0})\ar^{\pi_{0}}[d]\\ 
Z&Z_{0} \ar^{h}[l]
}
$$
is the desired diagram. 
So we are done. 
\end{step}
\end{proof}


\begin{rem}
By construction of the diagram we see that the divisor $B''_{0}$ is reduced, i.e., all coefficients of $B''_{0}$ are one (cf.~\cite[Lemma 4.5]{has-trivial}). 
But we do not use this fact in this paper.  
\end{rem}

\begin{lem}\label{lemlift}
Let $\pi\!:\!(X,B)\to Z$ be a contraction such that 
\begin{itemize}
\item
$(X,B)$ is a projective $\mathbb{Q}$-factorial lc pair such that $(X,0)$ is klt, 
\item 
$K_{X}+B\sim_{\mathbb{R}}\pi^{*}D$ for some $D$ on $Z$,
\item
$Z$ is a projective $\mathbb{Q}$-factorial variety such that $(Z,0)$ is klt, and
\item
$B=B'+B''$ with $B'\geq0$ and $B''\geq0$ such that $B''\sim_{\mathbb{R},\,Z}0$ and any lc center of $(X,B')$ dominates $Z$. 
\end{itemize}
Let $T$ be an effective $\mathbb{R}$-divisor on $Z$ such that $B''\sim_{\mathbb{R}}\pi^{*}T$. 
If $D$ is pseudo-effective but $D-eT$ is not pseudo-effective for any $e>0$, then we can construct the following diagram 
$$
\xymatrix{
(X,B)\ar@{-->}[r]\ar_{\pi}[d]&(\widetilde{X},\widetilde{B})\ar^{\widetilde{\pi}}[d]\\ 
Z\ar@{-->}[r]&\widetilde{Z}\ar[r]&Z^{\vee}
}
$$
such that 
\begin{itemize}
\item
$(\widetilde{X},\widetilde{B})$ is projective  $\mathbb{Q}$-factorial lc, $(X,0)$ is klt,   $\widetilde{Z}$ is projective and $\mathbb{Q}$-factorial, $(\widetilde{Z},0)$ is klt, and $Z^{\vee}$ is normal and projective,
\item
the maps $X\dashrightarrow \widetilde{X}$ and $Z\dashrightarrow\widetilde{Z}$ are birational contractions, 
\item
the morphism $\widetilde{Z}\to Z^{\vee}$ is a contraction such that  $\rho(\widetilde{Z}/Z^{\vee})=1$ and ${\rm dim}\,Z^{\vee}<{\rm dim}\,\widetilde{Z}$, and  
\item
$K_{\widetilde{X}}+\widetilde{B}\sim_{\mathbb{R}}\widetilde{\pi}^{*}\widetilde{D}$ and  $\widetilde{D}\sim_{\mathbb{R},\,Z^{\vee}}0$. 
\end{itemize}
Here the divisors $\widetilde{B}$ and $\widetilde{D}$ are the birational transform of $B$ on $\widetilde{X}$ and $D$ on $\widetilde{Z}$ respectively. 
\end{lem}

\begin{proof}
We can construct the desired diagram by the same argument as in \cite[Step 1 and 2 in the proof of Proposition 5.3]{has-trivial}. 
We write down the details for the reader's convenience. 

Let $\{e_{n}\}_{n\geq1}$ be a strictly decreasing sequence of positive real numbers such that $e_{n}<1$ for any $n$ and ${\rm lim}_{n\to \infty} e_{n}=0$. 
By \cite[Corollary 3.2]{fg-bundle}, for any $n\geq1$, we can find a boundary $\mathbb{R}$-divisor $\Theta_{n}$ such that $(Z,\Theta_{n})$ is klt and 
$$K_{X}+B-e_{n}B''\sim_{\mathbb{R}}\pi^{*}(D-e_{n}T)\sim_{\mathbb{R}}\pi^{*}(K_{Z}+\Theta_{n}).$$ 
Since $K_{Z}+\Theta_{n}\sim_{\mathbb{R}}D-e_{n}T$ is not pseudo-effective for any $n\geq1$, we can run the $(K_{Z}+\Theta_{n})$-MMP with scaling and obtain a Mori fiber space. 
Let $Z\dashrightarrow \widetilde{Z}_{n}$ be the birational contraction of a finitely many steps of the $(K_{Z}+\Theta_{n})$-MMP, and let $\widetilde{Z}_{n}\to Z_{n}^{\vee}$ be the contraction of the Mori fiber space. 
Let $\widetilde{D}_{n}$ and $\widetilde{T}_{n}$ be the birational transform of $D$ and $T$ on $\widetilde{Z}_{n}$ respectively.
Since $K_{Z}+\Theta_{n}\sim_{\mathbb{R}}D-e_{n}T$ and since $D$ is pseudo-effective, we see that  $\widetilde{D}_{n}-e_{n}\widetilde{T}_{n}$ is anti-ample over $Z_{n}^{\vee}$ and $\widetilde{T}_{n}$ is ample over $Z_{n}^{\vee}$. 
Furthermore, by applying the $\mathbb{R}$-boundary divisor version of \cite[Lemma 3.6]{has-trivial}, we have the following diagram 
$$
\xymatrix{
(X,B-e_{n}B'')\ar@{-->}[rr] \ar_{\pi}[d]&&(\widetilde{X}_{n},\widetilde{B}_{n}-e_{n}\widetilde{B}''_{n})\ar^{\pi_{n}}[d]\\ 
Z\ar@{-->}[rr]&&\widetilde{Z}_{n}\ar[r]&Z_{n}^{\vee}
}
$$
such that the upper horizontal birational map is a finitely many steps of the $(K_{X}+B-e_{n}B'')$-MMP and 
$$K_{\widetilde{X}_{n}}+\widetilde{B}_{n}-e_{n}\widetilde{B}''_{n}\sim_{\mathbb{R}}\pi_{n}^{*}(\widetilde{D}_{n}
-e_{n}\widetilde{T}_{n})\quad {\rm and}\quad \widetilde{B}''_{n}\sim_{\mathbb{R}}\pi_{n}^{*}\widetilde{T}_{n},$$
where $\widetilde{B}_{n}$ and $\widetilde{B}''_{n}$ are the birational transform of $B$ and $B''$ on $\widetilde{X}_{n}$. 
Now we apply 
Theorem \ref{thmacclct} to $X_{n}$ and apply Theorem \ref{thmglobalacc} to the general fiber of $\widetilde{X}_{n}\to Z_{n}^{\vee}$.  
Then we see that for some $n$ the pair $(\widetilde{X}_{n},\widetilde{B}_{n})$ is lc and $K_{\widetilde{X}_{n}}+\widetilde{B}_{n}\sim_{\mathbb{R},\,Z_{n}^{\vee}}0$ (cf.~\cite[Step 2 in the proof of Proposition 5.3]{has-trivial}). 
We also see that $\widetilde{D}_{n}\sim_{\mathbb{R},\,Z_{n}^{\vee}}0$ because we have  $K_{\widetilde{X}_{n}}+\widetilde{B}_{n}\sim_{\mathbb{R}}\pi_{n}^{*}\widetilde{D}_{n}$. 
For this $n$ we put $\widetilde{Z}=\widetilde{Z}_{n}$ and $Z^{\vee}=Z_{n}^{\vee}$. 
Then it is easy to see that the following
$$
\xymatrix{
(X,B)\ar@{-->}[r]\ar_{\pi}[d]&(\widetilde{X},\widetilde{B})\ar[d]\\ 
Z\ar@{-->}[r]&\widetilde{Z}\ar[r]&Z^{\vee}
}
$$
is the desired diagram. 
\end{proof}


\begin{proof}[Proof of Theorem \ref{thmmain}]
By hypothesis there is $C$ on $X$ such that $(X,C)$ is lc and $K_{X}+C\equiv0$. 
Then we have $K_{X}+C\sim_{\mathbb{R}}0$ by the abundance theorem for numerically trivial lc pairs. 
Therefore we may assume $C\neq0$ and Theorem \ref{thmmain} for $(X,\Delta)$ is equivalent to Theorem \ref{thmmain} for $(X,t\Delta+(1-t)C)$ for any $0<t\ll1$. 
So we will freely replace $(X,\Delta)$ with $(X,t\Delta+(1-t)C)$. 

By taking a dlt blow-up of $(X,C)$ and by replacing $(X,\Delta)$ with $(X,t\Delta+(1-t)C)$ for some $0<t\ll1$ we can assume $X$ is $\mathbb{Q}$-factorial and $(X,0)$ is klt. 
Since $C\neq0$, $K_{X}$ is not pseudo-effective, and thus $\tau(X,0;\Delta)>0$. 
Replacing $(X,\Delta)$ by $(X,\tau(X,0;\Delta)\Delta)$, we can assume that $\tau(X,0;\Delta)=1$. 

We prove Theorem \ref{thmmain} by induction on the dimension of $X$.

\begin{step3}\label{step1non} 
By \cite[Lemma 3.1]{gongyo-nonvanishing}, we can construct a birational contraction $\phi\!:\!X\dashrightarrow X'$ and a contraction $X'\to Z'$ such that ${\rm dim}\,Z'<{\rm dim}\,X$, $(X',\phi_{*}\Delta)$ is lc and $K_{X'}+\phi_{*}\Delta\sim_{\mathbb{R},\,Z'}0$. 
Then $(X',\phi_{*}C)$ is also lc since $K_{X}+C\sim_{\mathbb{R}}0$. 
Take a log resolution $Y \to X$ of $(X,{\rm Supp}(\Delta+C))$ so that the induced map $f\!:\!Y\dashrightarrow X'$ is a morphism, and let $(Y,\Delta_{Y})$ and $(Y,C_{Y})$ be log smooth models of $(X,\Delta)$ and $(X,C)$ respectively. 

Since $K_{X}+C\sim_{\mathbb{R}}0$, we see that 
$K_{Y}+C_{Y}-f^{*}(K_{X'}+\phi_{*}C)$
is effective and $f$-exceptional. 
So we can run the $(K_{Y}+C_{Y})$-MMP over $X'$ and get a model $f'\!:\!(Y',C_{Y'})\to X'$ such that $K_{Y'}+C_{Y'}=f'^{*}(K_{X'}+\phi_{*}C)\sim_{\mathbb{R}}0.$ 
By construction $(Y',C_{Y'})$ is lc and $Y\dashrightarrow Y'$ is a finitely many steps of the $(K_{Y}+t\Delta_{Y}+(1-t)C_{Y})$-MMP for any $0<t\ll1$. 
Fix a sufficiently small $t>0$ and set $\Gamma_{Y}=t\Delta_{Y}+(1-t)C_{Y}$. 
Let $\Gamma_{Y'}$ be the birational transform of $\Gamma_{Y}$ on $Y'$.  
Then we can write
$$K_{Y'}+\Gamma_{Y'}=f'^{*}\bigl(K_{X'}+t\phi_{*}\Delta+(1-t)\phi_{*}C\bigr)+F$$
with an $f'$-exceptional divisor $F$. 
Note that $F$ may not be effective. 
Run the $(K_{Y'}+\Gamma_{Y'})$-MMP over $X'$ with scaling. 
By \cite[Theorem 3.5]{birkar-flip} we reach a model $f''\!:\!(Y'',\Gamma_{Y''})\to X'$ such that 
$$K_{Y''}+\Gamma_{Y''}=f''^{*}\bigl(K_{X'}+t\phi_{*}\Delta+(1-t)\phi_{*}C\bigr)+F_{Y''}$$
with $F_{Y''}\leq0$. 
Now we recall that $(X',\phi_{*}\Delta)$ and $(X',\phi_{*}C)$ are both lc. 
Combining it with the above equation we see that $(Y'',\Gamma_{Y''}-F_{Y''})$ is also lc. 
By construction we also have $K_{Y''}+\Gamma_{Y''}-F_{Y''}\sim_{\mathbb{R},\,Z'}0$. 
Since $-F_{Y''}\geq 0$ and $(Y'',0)$ is $\mathbb{Q}$-factorial klt, by \cite[Theorem 1.1]{has-mmp}, we can run the $(K_{Y''}+\Gamma_{Y''})$-MMP over $Z'$ and obtain a good minimal model $(Y'',\Gamma_{Y''})\dashrightarrow (Y''',\Gamma_{Y'''})$ over $Z'$. 
Let $\pi\!:\!Y'''\to Z$ be the contraction over $Z'$ induced by $K_{Y'''}+\Gamma_{Y'''}$, and let  $C_{Y'''}$ be the birational transform of $C_{Y}$ on $Y'''$. 
Note that ${\rm dim}\,Z={\rm dim}\,Z'$ because $Z$ is birational to $Z'$. 
We also have $K_{Y'''}+\Gamma_{Y'''}\sim_{\mathbb{R},\,Z}0$ and $K_{Y'''}+C_{Y'''}\sim_{\mathbb{R}}0$. 
Furthermore, by construction, the birational map $Y\dashrightarrow Y'''$ is a finitely many steps of the $(K_{Y}+\Gamma_{Y})$-MMP. 
Therefore we can replace $(X,\Delta)$ and $(X,C)$ by $(Y''',\Gamma_{Y'''})$ and $(Y''',C_{Y'''})$. 

In this way, to prove Theorem \ref{thmmain}, we can assume that there exists a contraction $\pi\!:\!X\to Z$ to a normal projective variety $Z$ such that ${\rm dim}\,Z<{\rm dim}\,X$ and $K_{X}+\Delta\sim_{\mathbb{R},\,Z}0$. 
\end{step3}
 
 \begin{step3}\label{step2non} 
We apply Lemma \ref{lemdeform} to $(X,C)\to Z$ (not $(X,\Delta)\to Z$) and obtain a diagram
$$
\xymatrix{
(X,C)\ar@{-->}[r] \ar[d]_{\pi}&(X_{0},C_{0})\ar^{\pi_{0}}[d]\\ 
Z&Z_{0}\ar^{h}[l]
}
$$
such that
\begin{itemize}
\item
$\pi_{0}$ and $h$ are contractions and $h$ is birational, 
\item
$(X_{0},C_{0})$ is a log birational model of $(X,C)$ and it is a projective $\mathbb{Q}$-factorial lc pair such that $(X_{0},0)$ is klt,  
\item
$K_{X_{0}}+C_{0}\sim_{\mathbb{R}}0$, 
\item
$Z_{0}$ is a projective $\mathbb{Q}$-factorial variety and $(Z_{0},0)$ is klt, and
\item
$C_{0}=C'_{0}+C''_{0}$ with $C'_{0}\geq0$ and $C''_{0}\geq0$ such that $C''_{0}\sim_{\mathbb{R},\,Z_{0}}0$ and 
any lc center of $(X_{0},C'_{0})$ dominates $Z_{0}$. 
\end{itemize} 
Let $\varphi\!:\!W\to X$ and $\varphi_{0}\!:\!W\to X_{0}$ be a common resolution. 
We define a divisor $\Psi$ on $W$ by equation $K_{W}+\Psi=\varphi^{*}(K_{X}+\Delta)$ and set $\Delta_{0}=\varphi_{0*}\Psi$. 
Note that $\Delta_{0}$ may not be effective but $t\Delta_{0}+(1-t)C_{0}$ is effective for any $0<t\ll1$ because $(X_{0},C_{0})$ is a log birational model of $(X,C)$. 
By construction $K_{X_{0}}+\Delta_{0}\sim_{\mathbb{R},\,Z_{0}}0$ and any lc center of $(X_{0}, t\Delta_{0}+(1-t)C_{0})$ is an lc center of $(X_{0},C_{0})$. 
We can easily check that we can replace $(X,\Delta)\to Z$ and $(X,C)$ by $(X_{0},t\Delta_{0}+(1-t)C_{0})\to Z_{0}$ and $(X_{0},C_{0})$. 
Therefore we can assume that 
\begin{enumerate}
\item[(i)]
$Z$ is a projective $\mathbb{Q}$-factorial variety and $(Z,0)$ is klt, 
\item[(ii)]
$C=C'+C''$ for some $C'\geq0$ and $C''\geq0$ such that $C''\sim_{\mathbb{R},\,Z}0$ and 
any lc center of $(X,C')$ dominates $Z$, and 
\item[(iii)]
any lc center of $(X, \Delta)$ is an lc center of $(X,C)$. 
\end{enumerate}
\end{step3}

\begin{step3}\label{step3non}
In this step we prove Theorem \ref{thmmain} for $(X,\Delta)$ when $C''=0$. 
In this case we have $C=C'$. 

By conditions (ii) and (iii) in Step \ref{step2non}, all lc centers of $(X,\Delta)$ and those of $(X,C)$ dominate $Z$. 
Therefore, by \cite[Corollary 3.2]{fg-bundle}, there exists $\Theta$ (resp.~$G$) on $Z$ such that $(Z,\Theta)$ is klt (resp.~$(Z,G)$ is klt) and $K_{X}+\Delta\sim_{\mathbb{R}}\pi^{*}(K_{Z}+\Theta)$ (resp.~$K_{X}+C\sim_{\mathbb{R}}\pi^{*}(K_{Z}+G)$). 
Then there is $E\geq0$ such that $K_{Z}+\Theta\sim_{\mathbb{R}}E$ by induction hypothesis. 
Thus we see that  $K_{X}+\Delta\sim_{\mathbb{R}}\pi^{*}E$ and so we are done. 
\end{step3} 

\begin{step3}\label{step4non}
By Step \ref{step3non} we can assume that $C''\neq0$. 
Then $K_{X}+C'\sim_{\mathbb{R}}-C''$ is not pseudo-effective, and hence $K_{X}+t\Delta+(1-t)C-(1-t)C''$ is not pseudo-effective for any $0<t\ll1$. 
Moreover any lc center of $(X, t\Delta+(1-t)C')$ is an lc center of $(X,C')$. 
We fix a sufficiently small $t>0$ and we replace $(X,\Delta)$ by $(X,t\Delta+(1-t)C)$. 
We also see that we can replace $C''$ by $(1-t)C''$ (at the same time $C'$ is replaced by $C'+tC''$). 
Therefore replacing $C''$ we can assume that $\Delta-C''\geq0$, $K_{X}+\Delta-C''$ is not pseudo-effective, and any lc center of $(X, \Delta-C'')$ is an lc center of $(X,C')$. 
Then by condition (ii) in Step \ref{step2non} any lc center of $(X, \Delta-C'')$ dominates $Z$. 

Now we put $\tau=\tau(X,\Delta-C'';C'')$, where the right hand side is the pseudo-effective threshold of $C''$ with respect to $(X,\Delta-C'')$. 
By construction we have $0<\tau\leq1$. 
Therefore we can replace $(X,\Delta)$ by $(X,\Delta-C''+\tau C'')$. 
We can also replace $C''$ with $\tau C''$ and replace $C'$ with $C'+(1-\tau) C''$.  
Note that any lc center of $(X,C'+(1-\tau) C'')$ is an lc center of $(X,C')$ because $\tau>0$ and $(X,C)$ is lc. 

In this way, by replacing those divisors, we can assume that
\begin{itemize}
\item $\Delta-C''\geq0$ and any lc center of $(X,\Delta-C'')$ dominates $Z$, and
\item
$K_{X}+\Delta-eC''$ is not-pseudo-effective for any $e>0$.  
\end{itemize}
In the rest of the proof we do not use $C'$.  
\end{step3}

\begin{step3}\label{step5non}
Pick divisors $D$ and $T$ on $Z$ such that $K_{X}+\Delta\sim_{\mathbb{R}}\pi^{*}D$ and $C''\sim_{\mathbb{R}}\pi^{*}T$ respectively. 
By Step \ref{step1non}, \ref{step2non} and \ref{step4non}, $(X,\Delta)\to Z$ and $C''\neq0$ satisfy 
\begin{itemize}
\item 
$(X,\Delta)$ is a projective $\mathbb{Q}$-factorial lc pair such that $(X,0)$ is klt, 
\item
$K_{X}+\Delta\sim_{\mathbb{R}}\pi^{*}D$, 
\item
$Z$ is a projective $\mathbb{Q}$-factorial variety such that $(Z,0)$ is klt,  
\item
$\Delta-C''\geq0$, $C''\geq0$, $C''\sim_{\mathbb{R}}\pi^{*}T$ and any lc center of $(X,\Delta-C'')$ dominates $Z$, and  
\item
$K_{X}+\Delta-eC''$ is not-pseudo-effective for any $e>0$.   
\end{itemize}
Therefore we can apply Lemma \ref{lemlift} and we can obtain the following diagram 
$$
\xymatrix{
(X,\Delta)\ar@{-->}[r]\ar_{\pi}[d]&(\widetilde{X},\widetilde{\Delta})\ar^{\widetilde{\pi}}[d]\\ 
Z\ar@{-->}[r]&\widetilde{Z}\ar[r]&Z^{\vee}
}
$$
such that 
\begin{itemize}
\item
$(\widetilde{X},\widetilde{\Delta})$ is a projective  $\mathbb{Q}$-factorial lc pair,  $\widetilde{Z}$ is projective and $\mathbb{Q}$-factorial, and $Z^{\vee}$ is a normal projective variety,
\item
the maps $X\dashrightarrow \widetilde{X}$ and $Z\dashrightarrow\widetilde{Z}$ are birational contractions, 
\item
the morphism $\widetilde{Z}\to Z^{\vee}$ is a contraction such that  $\rho(\widetilde{Z}/Z^{\vee})=1$ and ${\rm dim}\,Z^{\vee}<{\rm dim}\,\widetilde{Z}$, and  
\item
$K_{\widetilde{X}}+\widetilde{\Delta}\sim_{\mathbb{R},\,Z^{\vee}}0$. 
\end{itemize}
Here $\widetilde{\Delta}$ is the birational transform of $\Delta$ on $\widetilde{X}$. 
We take a log resolution $Y_{1}\to X$ of $(X,{\rm Supp}\,(\Delta+C))$ such that the induced map $Y_{1}\dashrightarrow \widetilde{X}$ is a morphism. 
Let $(Y_{1},\Delta_{Y_{1}})$ and $(Y_{1},C_{Y_{1}})$ be log smooth models of $(X,\Delta)$ and $(X,C)$ respectively. 
Then we can apply the argument of Step \ref{step1non} to $Y_{1}\to \widetilde{X}\to Z^{\vee}$ since $(\widetilde{X},\widetilde{\Delta})$ is lc and $K_{\widetilde{X}}+\widetilde{\Delta}\sim_{\mathbb{R},\,Z^{\vee}}0$. 
Thus we can get a contraction $Y'''_{1}\to Z_{1}$ over $Z^{\vee}$ and lc pairs $(Y'''_{1},\Gamma_{Y'''_{1}})$ and $(Y'''_{1},C_{Y'''_{1}})$ such that $K_{Y'''_{1}}+\Gamma_{Y'''_{1}}\sim_{\mathbb{R},\,Z_{1}}0$ and $K_{Y'''_{1}}+C_{Y'''_{1}}\sim_{\mathbb{R}}0$. 
Here $C_{Y'''_{1}}$ is  the birational transform of $C_{Y_{1}}$ on $Y'''_{1}$ and $\Gamma_{Y'''_{1}}$ is the birational transform of $t\Delta_{Y_{1}}+(1-t)C_{Y_{1}}$ on $Y'''_{1}$ for a sufficiently small $t>0$. 
Furthermore we can check that we may replace $(X,\Delta)\to Z$ and $(X,C)$ by $(Y'''_{1},\Gamma_{Y'''_{1}})\to Z_{1}$ and $(Y'''_{1},C_{Y'''_{1}})$. 
For details, see the second paragraph of Step \ref{step1non}. 

We replace $(X,\Delta)\to Z$ by $(Y'''_{1},\Gamma_{Y'''_{1}})\to Z_{1}$. 
Then the dimension of $Z$ is strictly decreased. 
This is crucial to the proof. 
\end{step3}

\begin{step3}
From now on we repeat the argument of Step \ref{step2non}-\ref{step5non}. 

By the same argument as in Step \ref{step2non}, we can assume $(X,\Delta)\to Z$ and $(X,C)$ satisfy conditions (i), (ii) and (iii) in Step \ref{step2non}. 
Then there are two possibilities: 
\begin{itemize}
\item
Theorem \ref{thmmain} holds for $(X,\Delta)$ (cf.~Step \ref{step3non}), or
\item
we can find a contraction $Y'''_{2}\to Z_{2}$ with ${\rm dim}\,Z_{2}<{\rm dim}\,Z$ and lc pairs $(Y'''_{2},\Gamma_{Y'''_{2}})$ and $(Y'''_{2},C_{Y'''_{2}})$ such that $K_{Y'''_{2}}+\Gamma_{Y'''_{2}}\sim_{\mathbb{R},\,Z_{2}}0$, 
$K_{Y'''_{2}}+C_{Y'''_{2}}\sim_{\mathbb{R}}0$ and Theorem \ref{thmmain} for $(X,\Delta)$ is implied from Theorem \ref{thmmain} for $(Y'''_{2},\Gamma_{Y'''_{2}})$ (cf.~Step \ref{step4non} and \ref{step5non}).  
\end{itemize}
If we are in the first case we stop the argument. 
If we are in the second case we replace $(X,\Delta)\to Z$ by $(Y'''_{2},\Gamma_{Y'''_{2}})\to Z_{2}$ and repeat the argument of Step \ref{step2non}-\ref{step5non}. 
Each time we replace $(X,\Delta)\to Z$ in the argument of Step \ref{step5non}, the dimension of $Z$ is strictly decreased. 
Therefore this discussion eventually stops. 
Thus we can prove Theorem \ref{thmmain} and so we are done. 
\end{step3}
\end{proof}


\end{document}